\documentclass[12pt,reqno]{amsart}
\usepackage{mathrsfs}\overfullrule=5pt
\usepackage{amsfonts,epsfig,amsmath,amssymb,amscd,graphicx}
\usepackage[colorlinks=true]{hyperref}
\hypersetup{urlcolor=red, citecolor=blue}
\usepackage{pstricks}
\input xy
\xyoption{all}

\newtheorem{thm}{Theorem}[section]
\newtheorem{lem}[thm]{Lemma}

\newtheorem{prop}[thm]{Proposition}

\theoremstyle{definition}

\theoremstyle{remark}
\newtheorem{rem}[thm]{Remark}

\numberwithin{equation}{section}

\def \N {\mathbb N}
\def \R {\mathbb R}

\begin{document}

\title{Topological entropy of nonautonomous dynamical systems}
\author[K. Liu]{Kairan Liu}
\address{Kairan Liu:
Department of Mathematics, University of Science and Technology of China, Hefei, Anhui 230026, China}
\email{lkr111@mail.ustc.edu.cn}
\author[Y. Qiao]{Yixiao Qiao $^*$}{\let\thefootnote\relax\footnote{* Corresponding author.}}
\address{Yixiao Qiao:
School of Mathematical Sciences, South China Normal University, Guangzhou, Guangdong 510631, China}
\email{yxqiao@impan.pl}
\author[L. Xu]{Leiye Xu}
\address{Leiye Xu:
Department of Mathematics, University of Science and Technology of China, Hefei, Anhui 230026, China}
\email{leoasa@mail.ustc.edu.cn}


\begin{abstract}
Let $\mathcal{M}(X)$ be the space of all Borel probability measures on a compact metric space $X$ endowed with the weak$^\ast$-topology. In this paper, we prove that if the topological entropy of a nonautonomous dynamical system $(X,\{f_n\}_{n=1}^{+\infty})$ vanishes, then so does that of its induced system $(\mathcal{M}(X),\{f_n\}_{n=1}^{+\infty})$; moreover, once the topological entropy of $(X,\{f_n\}_{n=1}^{+\infty})$ is positive, that of its induced system $(\mathcal{M}(X),\{f_n\}_{n=1}^{+\infty})$ jumps to infinity. In contrast to Bowen's inequality, we construct a nonautonomous dynamical system whose topological entropy is not preserved under a finite-to-one extension.

\end{abstract}

\maketitle

\section{Introduction}
As an important invariant of topological conjugacy, the notion of topological entropy was introduced by Adler, Konheim and McAndrew \cite{Adler-Konheim-McAndrew} in 1965. Topological entropy is a key tool to measure the complexity of a classical dynamical system, i.e. the exponential growth rate of the number of distinguishable orbits of the iterates of an endomorphism of a compact metric space. In order to have a good understanding of the topological entropy of a skew product of dynamical systems (as we know that the calculation of its topological entropy can be transformed into that of its fibers), Kolyada and Snoha \cite{Kolyada-Snoha} proposed the concept of topological entropy in 1966 for a nonautonomous dynamical system determined by a sequence of maps.

A \textbf{nonautonomous dynamical system} (NADS for short) is a pair $(X,\{f_n\}_{n=1}^{+\infty})$, where $X$ is a compact metric space endowed with a metric $\rho$ and $\{f_n\}_{n=1}^{+\infty}$ is a sequence of continuous maps from $X$ to $X$. In 2013, Kawan \cite{Kawan} generalized the classical notion of measure-theoretical entropy established by Kolmogorov and Sinai to NADSs, and proved that the measure-theoretical entropy can be estimated from above by its topological entropy. Following the idea of Brin and Katok \cite{Brin-Katok}, Xu and Zhou \cite{Xu-Zhou} introduced the measure-theoretical entropy in nonautonomous case and established a variational principle for the first time. More results related to entropy for NADSs were developed in \cite{Bis,Canovas,Huang-Wen-Zeng,Kawan,Kolyada-Misiurewicz-Snoha,Zhang-Chen,Zhu-Liu-Zhang}.

In contrast to the classical dynamical systems whose dynamics have been fully studied, properties of entropy for NADSs are still fairly poor-developed. One of such respects that we considered naturally is the relation between a NADS and its induced system (whose phase space consists of all Borel probability measures on the original space, for details see Section \ref{sec:pre}). A well-known result due to Glasner and Weiss \cite{Glasner-Weiss} in 1995 reveals that if a system has zero topological entropy, then so does its induced system. This theorem is amazing. Generally speaking, a system is rather ``tiny'' (in the sense of a subsystem) compared with its induced system. However, the vanishment of its entropy surprisingly results in the same phenomenon for its induced system. Later, this connection was further developed by Kerr and Li in \cite{Kerr-Li}. They obtained that a system is null if and only if its induced system is null (recall that a classical dynamical system is null if its topological sequence entropy along any increasing positive sequence is zero). In \cite{Qiao-Zhou}, the second named author and Zhou generalized the result of Glasner and Weiss to any increasing positive sequence for classical dynamical systems. This generalization strengthens Kerr and Li's result as well.

The present paper aims to investigate the entropy relation between a system and its induced system in the context of NADSs. We denote by $\mathcal{M}(X)$ the space of all Borel probability measures on a compact metric space $X$ equipped with the weak$^\ast$-topology. Our main result is as follows.

\begin{thm}\label{thm-A}
Let $(X,\{f_n\}_{n=1}^{+\infty})$ be a NADS. Then the following statements hold:
\begin{enumerate}
  \item $h_{top}(X,\{f_n\}_{n=1}^{+\infty})=0$ if and only if $h_{top}(\mathcal{M}(X),\{f_n\}_{n=1}^{+\infty})=0$.
  \item $h_{top}(X,\{f_n\}_{n=1}^{+\infty})>0$ if and only if $h_{top}(\mathcal{M}(X),\{f_n\}_{n=1}^{+\infty})=+\infty$.
\end{enumerate}
\end{thm}

Note that Theorem \ref{thm-A} includes the results mentioned previously in \cite{Glasner-Weiss,Qiao-Zhou}.

\medskip

Now let us turn to considering the entropy relation between a system and its extensions. In classical dynamical systems, topology entropy, as we know, is preserved under finite-to-one extensions \cite{Bowen}. A natural question is if we may further expect such an assertion to be true for NADSs. Unfortunately, this property fails in nonautonomous case.

\begin{thm}\label{exam:extension}
There exist two NADSs $(X,\{f_n\}_{n=1}^{+\infty})$ and $(Y,\{g_n\}_{n=1}^{+\infty})$ such that $(X,\{f_n\}_{n=1}^{+\infty})$ is a finite-to-one extension of $(Y,\{g_n\}_{n=1}^{+\infty})$ but $h_{top}(X,\{f_n\}_{n=1}^{+\infty})>h_{top}(Y,\{g_n\}_{n=1}^{+\infty})$.
\end{thm}

Theorem \ref{exam:extension} reflects that entropy properties of NADSs may differ from that of classical dynamical systems. In particular, it indicates that the Bowen-type entropy inequality (stated in Theorem \ref{thm:finitetoone}) does not hold for NADSs in general.

\medskip

This paper is organized as follows. In Section \ref{sec:pre}, we list basic notions and results needed in our argument. In Section \ref{sec:prop}, we prove Theorem \ref{thm-A}(2). In Section \ref{sec:mainthm}, we prove Theorem \ref{thm-A}(1). In Section \ref{sec:exam}, we provide a constructive proof of Theorem \ref{exam:extension}.

\medskip

\textbf{Acknowledgements.}
The authors would like to thank Prof. Wen Huang and Dr. Lei Jin for their useful comments and suggestions. Y. Qiao was partially supported by NSFC of China (11871228). L. Xu was partially supported by NNSF of China (11801538, 11871188) and the Fundamental Funds for the Central Research Universities.

\section{Preliminaries}\label{sec:pre}

For clarification, throughout this paper by a \textbf{topological dynamical system} (TDS for short) we mean a pair $(X,T)$, where $X$ is a compact metric space endowed with a metric $\rho$ and $T: X\to X$ is a homeomorphism. A \textbf{nonautonomous dynamical system} (NADS for short) is a pair $(X,\{f_n\}_{n=1}^{+\infty})$, where $X$ is a compact metric space endowed with a metric $\rho$ and $\{f_n:X\to X\}_{n=1}^{+\infty}$ is a sequence of continuous maps. We denote by $\N$ and $\N_+$ the sets of nonnegative integers and positive integers, respectively.

\subsection{Topological entropy}
Let $(X,\{f_n\}_{n=1}^{+\infty})$ be a NADS and $\rho$ a metric on $X$. An \textbf{open cover} of $X$ is a family of open subsets of $X$, whose union is $X$. For two covers $\mathcal U$ and $\mathcal V$ we say that $\mathcal U$ is a \textbf{refinement} of $\mathcal V$ if for each $U\in\mathcal{U}$ there is $V\in\mathcal{V}$ with $U\subset V$. For $n\in\N_+$ and open covers $\mathcal{U}_1,\mathcal{U}_2,\dots,\mathcal{U}_n$ of $X$ we denote
$$\bigvee_{i=1}^{n}\mathcal{U}_i=\left\{A_1\cap A_1\cap\dots\cap A_n: A_1\in\mathcal{U}_1,A_2\in\mathcal{U}_2,\dots,A_n\in\mathcal{U}_n\right\}.$$
Note that $\bigvee_{i=1}^n\mathcal{U}_i$ is also an open cover of $X$. We denote by $\mathcal{N}(\mathcal{U})$ the minimal cardinality of all subcovers chosen from $\mathcal{U}$.
Set $$f_i^0=\mathrm{id}_X,\;\; f_i^n=f_{i+(n-1)}\circ f_{i+(n-2)}\circ\dots\circ f_{i+1}\circ f_{i},\;\;f_i^{-n}=(f_i^{n})^{-1}$$
for all $i,n\in\N_+$, where $\mathrm{id}_X$ is the identity map on $X$. Let
$$h_{top}(\{f_n\}_{n=1}^{+\infty},\mathcal{U})=\limsup\limits_{n\to+\infty}
\frac{\log{\mathcal{N}(\bigvee_{j=0}^{n-1}f_1^{-j}(\mathcal{U}))}}{n}.$$
The \textbf{topological entropy} of $(X,\{f_n\}_{n=1}^{+\infty})$ is defined by
$$h_{top}(X,\{f_n\}_{n=1}^{+\infty})=\sup\left\{h_{top}(\{f_n\}_{n=1}^{+\infty},\mathcal{U}): \mathcal{U}\text{ is an open cover of } X\right\}.$$

As we expected, there is a Bowen-like equivalent definition of topological entropy for NADSs. For each $n\in\N_+$, a compatible metric $\rho_n$ on $X$ is defined  by the formula $$\rho_{n}(x,y)=\max_{0\leq j\leq n-1}\rho(f_1^jx,f_1^jy).$$
For any $n\in\N_+$ and $\varepsilon>0$, a subset $F$ of $X$ is called an \textbf{$(n,\varepsilon)$-spanning} subset of $(X,\{f_n\}_{n=1}^{+\infty})$ if for any $x\in X$ there exists $y\in F$ with $\rho_{n}(x,y)<\varepsilon$. A subset $E$ of $X$ is called an \textbf{$(n,\varepsilon)$-separated} subset of $(X,\{f_n\}_{n=1}^{+\infty})$ if for any distinct $x,y\in E$, $\rho_n(x,y)>\varepsilon$. We denote by $r_n(X,\{f_n\}_{n=1}^{+\infty},\varepsilon)$ the smallest cardinality of all $(n,\varepsilon)$-spanning subsets of $(X,\{f_n\}_{n=1}^{+\infty})$, and $s_n(X,\{f_n\}_{n=1}^{+\infty},\varepsilon)$ the largest cardinality of all $(n,\varepsilon)$-separated subsets of $(X,\{f_n\}_{n=1}^{+\infty})$. It was proved in \cite[Lemma 3.1]{Kolyada-Snoha} that for every NADS $(X,\{f_n\}_{n=1}^{+\infty})$, we have
\begin{align*}
h_{top}(X,\{f_n\}_{n=1}^{+\infty})& =\lim\limits_{\varepsilon\to 0}\limsup\limits_{n\to+\infty}\frac{\log{s_n(X,\{f_n\}_{n=1}^{+\infty},\varepsilon)}}{n}\\
&=\lim\limits_{\varepsilon\to 0}\limsup\limits_{n\to+\infty}\frac{\log{r_n(X,\{f_n\}_{n=1}^{+\infty},\varepsilon)}}{n}.
\end{align*}

\subsection{Extensions}
Let $(X,T)$ and $(Y,S)$ be two TDSs. We say that $(X,T)$ is an \textbf{extension} of $(Y,S)$ if there is a continuous surjective map $\pi: X\to Y$ such that $\pi\circ T=S\circ \pi$. For two NADSs $(X,\{f_n\}_{n=1}^{+\infty})$ and $(Y,\{g_n\}_{n=1}^{+\infty})$, $(X,\{f_n\}_{n=1}^{+\infty})$ is said to be an \textbf{extension} of $(Y,\{g_n\}_{n=1}^{+\infty})$ if there is a continuous surjective map $\pi: X\to Y$ such that $\pi\circ f_{n}=g_{n}\circ \pi$ for every $n\geq 1$. In both of the above definitions, $\pi$ is called an \textbf{extension} (or a \textbf{factor map}), and if in addition, there exists $c>0$ such that $\sup_{y\in Y}\#\pi^{-1}(y)\le c$, then $\pi$ is called \textbf{finite-to-one}.

It is easy to see that if $(X,T)$ is an extension of $(Y,S)$ then $h_{top}(X,T)\geq h_{top}(Y,S)$. Bowen \cite{Bowen} gave an upper bound of extensions in his renowned work as follows.

\begin{thm}[{\cite[Theorem 17]{Bowen}}]\label{thm:finitetoone}
Let $(X,T)$ and $(Y,S)$ be two TDSs, and $\pi: (X,T)\to (Y,S)$ an extension. Then
$$h_{top}(X,T)\leq h_{top}(Y,S)+\sup\limits_{y\in Y}h_{top}(T,\pi^{-1}(y)).$$
In particular, if $\pi$ is finite-to-one, then $h_{top}(X,T)=h_{top}(Y,S)$.
\end{thm}

\begin{rem}
In the case of NADSs, the assumption that for any $n\in\N_+$, $f_n$ is topologically conjugate to $g_n$ (via a homeomorphism $\pi_{n}: X\to Y$) is not sufficient to guarantee the equality $h_{top}(X,\{f_n\}_{n=1}^{+\infty})=h_{top}(Y,\{g_n\}_{n=1}^{+\infty})$. However, if all $\pi_n$'s are the same, then $h_{top}(X,\{f_n\}_{n=1}^{+\infty})=h_{top}(Y,\{g_n\}_{n=1}^{+\infty})$ holds (see \cite[Section 5.b]{Kolyada-Snoha}).
\end{rem}

\subsection{Induced systems}
Let $X$ be a compact metric space, $\mathcal{B}(X)$ the set of Borel subsets of $X$, $C(X)$ the space of continuous maps from $X$ to $\R$ endowed with the supremum norm ($||\cdot||_{\infty}$), and $\mathcal{M}(X)$ the set of Borel probability measures on $X$. The \textbf{weak$^\ast$-topology} is the smallest topology making the map
$$D_g:\mathcal{M}(X)\to\mathbb{R},\;\;\;\mu\mapsto\int_Xgd\mu$$ continuous for every $g\in C(X)$. A basis is given by the collection of all sets of the form
$$V_{\mu}(g_1,g_2,\dots,g_k;\varepsilon)=\left\{\nu\in\mathcal{M}(X):\left|\int g_id\mu-\int g_id\nu\right|<\varepsilon,\ 1\leq i\leq k\right\},$$
where $\mu\in \mathcal{M}(X)$, $g_1,g_2,\dots,g_k\in C(X)$, $k\in\N$ and $\varepsilon>0$. It is well known that $\mathcal{M}(X)$ is compact in the weak$^\ast$-topology \cite[Theorem 6.5]{Walters}.

Suppose that $\{g_n\}_{n=1}^{+\infty}$ is a dense subset of $C(X)$. By \cite[Theorem 6.4]{Walters}, the metric
$$D(\mu,\nu)=\sum\limits_{n=1}^{+\infty}\frac{|\int g_nd\mu-\int g_nd\nu|}{2^n(||g_{n}||_{\infty}+1)}$$
on $\mathcal{M}(X)$ is compatible with the weak$^\ast$-topology. So $\mathcal{M}(X)$ becomes a compact metric space as well.

A NADS $(X,\{f_n\}_{n=1}^{+\infty})$ induces a system $(\mathcal{M}(X),\{f_n^{\ast}\}_{n=1}^{+\infty})$, where $f_n^{\ast}:\mathcal{M}(X)\to \mathcal{M}(X)$ is given by $f_n^{\ast}(\mu)(B)=\mu(f_n^{-1}B)$ for each $n\in\N$, $\mu\in \mathcal{M}(X)$ and $B\in \mathcal{B}(X)$. We call $(\mathcal{M}(X),\{f_n^{\ast}\}_{n=1}^{+\infty})$ the \textbf{induced system} of $(X,\{f_n\}_{n=1}^{+\infty})$ and write $(\mathcal{M}(X),\{f_n\}_{n=1}^{+\infty})$ instead of $(\mathcal{M}(X),\{f_n^{\ast}\}_{n=1}^{+\infty})$ if there is no ambiguity.

\section{Proof of Theorem \ref{thm-A}(2)}\label{sec:prop}
Let $X$ be a compact metric space with the metric $\rho$ and $n\in\N_+$. The metric
$$d((x_1,x_2\dots,x_n),(y_1,y_2,\dots,y_n))=\max_{1\leq i\leq n}\rho(x_i,y_i)$$
on $X^n$ is compatible with the product topology. For a map $f: X\to X$, set
$$f^{(n)}=\underbrace{f\times f\times\cdots\times f}_{n\;\text{times}}:X^{n}\to X^n,\;\;\;(x_1,x_2,\dots,x_n)\mapsto (fx_1,fx_2,\dots,fx_n).$$
\begin{prop}\label{times}
Let $(X,\{f_n\}_{n=1}^{+\infty})$ be a NADS and $k\in\N_+$. Then
$$h_{top}(X^k,\{f_{n}^{(k)}\}_{n=1}^{+\infty})=k\cdot h_{top}(X,\{f_n\}_{n=1}^{+\infty}).$$
\end{prop}
\begin{proof} For fixed $m\in\N$ and $\varepsilon>0$, we let $E$ be an $(m,\varepsilon)$-spanning set of $(X,\{f_n\}_{n=1}^{+\infty})$ with $\#E=r_m(X,\{f_n\}_{n=1}^{+\infty},\varepsilon)$. Then for any $x=(x_1,x_2,\dotsc,x_k)\in X^k$, there exists $y=(y_1,y_2,\dotsc y_k)\in E^k$ such that $\rho_m(x_i,y_i)<\varepsilon$ for $i=1,2,\dots,k$.
Thus,
\begin{align*}\rho_m(x,y)&=\max_{0\leq j\leq m-1}d\left((f_1^jx_{1},\dotsc,f_1^jx_{k}),(f_1^jy_{1},\dotsc,f_1^jy_{k})\right)\\
&=\max_{0\leq j\leq m-1}\max_{1\leq i\leq k}\rho(f_1^jx_i,f_1^jy_i)\\
&=\max_{1\leq i\leq k}\rho_m(x_i,y_i)\\
&<\varepsilon.
\end{align*}
This implies that $E^{k}$ is an $(m,\varepsilon)$-spanning set of $X^{k}$, and hence
$$r_m(X^k,\{f^{(k)}_{n}\}_{n=1}^{+\infty},\varepsilon)\leq \#(E^{k})=(r_m(X,\{f_n\}_{n=1}^{+\infty},\varepsilon))^k$$
for any $m\in\N $ and $\varepsilon>0$. Therefore,
\begin{align}\label{entropy1}
h_{top}(X^k,\{f^{(k)}_{n}\}_{n=1}^{+\infty})&=\lim_{\varepsilon\to 0}\limsup_{m\to +\infty}\frac{1}{m}\log r_m(X^k,\{f^{(k)}_{n}\}_{n=1}^{+\infty},\varepsilon)\nonumber\\
&\leq \lim_{\varepsilon\to 0}\limsup_{m\to +\infty}\frac{k}{m}\log r_m(X,\{f_n\}_{n=1}^{+\infty},\varepsilon)\nonumber\\
&=k\cdot h_{top}(X,\{f_n\}_{n=1}^{+\infty}).
\end{align}

For fixed $n'\in\N $ and $\varepsilon'>0$, we assume that $F$ is an $(n',\varepsilon')$-separated set of $(X,\{f_n\}_{n=1}^{+\infty})$ with $\#F=s_{n'}(X,\{f_n\}_{n=1}^{+\infty},\varepsilon')$. For any two distinct points $x=(x_1,x_2,\dotsc,x_k)$ and
$y=(y_1,y_2,\dotsc,y_k)$ in $F^k$, we have
\begin{align*}
d_{n'}(x,y)&=\max_{0\leq j\leq n'-1}d\left((f_1^jx_{1},\dotsc,f_1^jx_{k}),(f_1^jy_{1},\dotsc,f_1^jy_{k})\right)\\
&=\max_{0\leq j\leq n'-1}\max_{1\leq i\leq k}\rho(f_1^jx_i,f_1^jy_i)\\
&=\max_{1\leq i\leq k}\rho_{n'}(x_i,y_i)\\
&>\varepsilon'.
\end{align*}
So $F^k$ is an $(n',\varepsilon')$-separated set of $(X^k,\{f^{(k)}_{n}\}_{n=1}^{+\infty})$, which means that
$$s_{n'}(X^k,\{f^{(k)}_{n}\}_{n=1}^{+\infty},\varepsilon')\geq \#(F^{k})=(s_{n'}(X,\{f_n\}_{n=1}^{+\infty},\varepsilon'))^k$$ for any $n'\in\N$ and $\varepsilon'>0$. Thus,
\begin{align}\label{entropy2}
h_{top}(X^k,\{f^{(k)}_{n}\}_{n=1}^{+\infty}) &=\lim_{\varepsilon'\to 0}
\limsup_{n'\to+\infty}\frac{1}{n'}\log s_{n'}(X^k,\{f^{(k)}_{n}\}_{n=1}^{+\infty},\varepsilon')\nonumber\\
&\geq \lim_{\varepsilon'\to 0}\limsup_{n'\to+\infty}\frac{k}{n'}\log s_{n'}(X,\{f_n\}_{n=1}^{+\infty},\varepsilon')\nonumber\\
&=k\cdot h_{top}(X,\{f_n\}_{n=1}^{+\infty}).
\end{align}
By \eqref{entropy1} and \eqref{entropy2}, we get $h_{top}(X^k,\{f^{(k)}_{n}\}_{n=1}^{+\infty})=k\cdot h_{top}(X,\{f_n\}_{n=1}^{+\infty})$.
\end{proof}

\begin{prop}\label{injective}
Let $X$ be a compact metric space and $k\in\N_+$. Then the map
$$\pi_k:X^k\to\mathcal{M}(X),\;\; (x_1,x_2,\dotsc,x_k)\mapsto\frac{1}{\sum_{i=1}^{k}2^{i}}\sum\limits_{i=1}^{k}2^i\delta_{x_i}$$
is injective.
\end{prop}
\begin{proof} Fix $x=(x_1,x_2,\dotsc,x_k)$ and $y=(y_1,y_2,\dotsc,y_k)$ in $X^{k}$. Set
$$t=\min\{i =1,2,\dotsc,k:x_i\ne y_i\}.$$ There exists a continuous function $g\in C(X)$ satisfying that $g(x_t)=1$ and that $g(z)=0$ for all $z\in\{x_1,x_2,\dotsc,x_k,y_1,y_2,\dotsc,y_k\}\setminus\{x_t\}$.
Then we have
$$\int gd\left(\sum\limits_{i=1}^{k}2^i\delta_{x_i}\right)=\int gd\left(\sum\limits_{i=1}^{t-1}2^i\delta_{x_i}\right) + \int gd\left(\sum\limits_{i=t}^{k}2^i\delta_{x_i}\right)$$
and $$\int gd\left(\sum\limits_{i=1}^{k}2^i\delta_{y_i}\right)=\int gd\left(\sum\limits_{i=1}^{t-1}2^i\delta_{y_i}\right)+\int gd\left(\sum\limits_{i=t}^{k}2^i\delta_{y_i}\right).$$
If $t=k$, then $$\int gd\left(\sum\limits_{i=t}^{k}2^i\delta_{x_i}\right)=2^k\neq0=\int gd\left(\sum\limits_{i=t}^{k}2^i\delta_{y_i}\right).$$ Otherwise, we have $$2^{t+1}\nmid \int gd\left(\sum\limits_{i=t}^{k}2^i\delta_{x_i}\right),\quad 2^{t+1}\mid \int gd\left(\sum\limits_{i=t}^{k}2^i\delta_{y_i}\right).$$ Summing up, $$\int gd(\sum\limits_{i=1}^{k}2^i\delta_{x_i})\neq\int gd(\sum\limits_{i=1}^{k}2^i\delta_{y_i}).$$ This implies $$\sum\limits_{i=1}^{k}2^i\delta_{x_i}\neq \sum\limits_{i=1}^{k}2^i\delta_{y_i}.$$ Thus, $\pi_k$ is injective.
\end{proof}
We are now ready to prove Theorem \ref{thm-A}(2).

\medskip

For any $k\in\N_+$, let $\pi_k$ be the map defined in Proposition \eqref{injective}. It is clear that $\pi_k$ is continuous and equivariant, which, together with the injectivity of $\pi_k$ that we just proved in Proposition \eqref{injective}, allows us to regard $(X^k,\{f_n^{(k)}\}_{n=1}^{+\infty})$ as a subsystem of $(\mathcal{M}(X),\{f_n\}_{n=1}^{+\infty})$. This implies that
$$h_{top}(\mathcal{M}(X),\{f_n\}_{n=1}^{+\infty})\ge h_{top}(X^k,\{f_n^{(k)}\}_{n=1}^{+\infty})=k\cdot h_{top}(X,\{f_n\}_{n=1}^{+\infty})$$
for all $k\in\N_+$. Since $h_{top}(X,\{f_n\}_{n=1}^{+\infty})>0$, we conclude $$h_{top}(\mathcal{M}(X),\{f_n\}_{n=1}^{+\infty})=+\infty.$$

\medskip

\section{Proof of Theorem \ref{thm-A}(1)}\label{sec:mainthm}
To begin with, we borrow a key lemma which has some combinatorial flavor.

\begin{lem}[{\cite[Proposition 2.1]{Glasner-Weiss}}]\label{lemma1}
For given constants $\varepsilon>0$ and $b>0$, there exist $n_0\in\N$ and a constant $c>0$ such that for all $n>n_0$, if $\phi$ is a linear mapping from $l_1^m$ to $l_{\infty}^{n}$ of norm
$$||\phi||=\sup\left\{||\phi(x)||_{\infty}:x\in l_1^m,||x||\leq 1\right\}\leq 1,$$
and if $\phi(B_1(l_1^m))$ contains more than $2^{bn}$ points that are $\varepsilon$-separated, then $m\geq 2^{cn}$, where $B_1(l_1^m):=\{y\in l_1^m:||y||\leq 1\}$.
\end{lem}

\medskip

Firstly, $(X,\{f_n\}_{n=1}^{+\infty})$ may be regarded as a subsystem of $(\mathcal{M}(X),\{f_n\}_{n=1}^{+\infty})$ by mapping $x\in X$ to $\delta_x\in\mathcal{M}(X)$, where
\[\delta_x(A)=\begin{cases}
1,& \text{if}\ x\in A \\
0,& \text{if}\ x\notin A
\end{cases}.\]
So $h_{top}(\mathcal{M}(X),\{f_n\}_{n=1}^{+\infty})=0$ implies $h_{top}(X,\{f_n\}_{n=1}^{+\infty})=0$.
\medskip

Now we assume $h_{top}(\mathcal{M}(X),\{f_n\}_{n=1}^{+\infty})>0$. We shall show $h_{top}(X,\{f_n\}_{n=1}^{+\infty})>0$. Let $\{g_n\}_{n=1}^{+\infty}$ be a sequence in $C(X)$ satisfying that $||g_n||\le 1$ for any $n\in\N_+$, and that
$$D(\mu,\nu)=\sum\limits_{n=1}^{+\infty}\frac{|\int g_n d\mu-\int g_nd\nu|}{2^n}$$
is a metric on $\mathcal{M}(X)$ giving the weak$^\ast$-topology.

Since $h_{top}(\mathcal{M}(X),\{f_n\}_{n=1}^{+\infty})>0$, there exist $a>0$ and $\varepsilon_0>0$ such that for any $0<\varepsilon<\varepsilon_0$ we can find an increasing sequence $\{N_i\}_{i=1}^{+\infty}\subset\N$ with $$s_{N_i}(\mathcal{M}(X),\{f_n\}_{n=1}^{+\infty},\varepsilon)>e^{aN_i}.$$
For any fixed $0<\varepsilon<\varepsilon_0$, there exists $K_0\in\N$ such that $\sum_{n=K_0+1}^{+\infty}1/2^n<\varepsilon/2$.
Since $g_n$ is continuous for any $n\in\N_+$, there exists $\delta>0$ such that $d(x,y)<\delta$ implies $d(g_n(x),g_n(y))<\varepsilon/9$, for all $x,y\in X$ and $n=1,2,\dotsc,K_0$.

Let $\mathcal{U}=\{U_1,U_2,\dotsc,U_d\}$ be an open cover of $X$ with $\text{diam}(\mathcal{U})<\delta$ and set $L_{N_i}=\mathcal{N}\left(\bigvee\limits_{j=0}^{N_i-1}f_1^{-j}\mathcal{U}\right)$. By the definition, we can take a subcover $\mathcal{V}=\{V_1,V_2,\dots,V_{L_{N_i}}\}$ of $\bigvee\limits_{j=0}^{N_i-1}f_{1}^{-j}\mathcal{U}$ of the minimal cardinality $L_{N_i}$. Set
$$A_1=V_1,\ A_2=V_2\setminus V_1,\dotsc,\ A_{L_{N_i}}=V_{L_{N_i}}\setminus\bigcup\limits_{j=1}^{L_{N_i}-1}V_j.$$
Then $\{A_1,A_2,\dotsc,A_{L_{N_i}}\}$ is a partition of $X$ and $A_i\neq\emptyset$ for every $i=1,2,\dotsc,L_{N_i}$. For each $i=1,2,\dotsc,L_{N_i}$ we take $z_i\in A_i$. Define $\phi: l_1^{L_{N_i}}\to l_{\infty}^{K_0\cdot N_i}$ by
$$\phi(\{x_k\}_{k=1}^{L_{N_i}})=\left\{\frac{1}{2^n}\sum\limits_{k=1}^{L_{N_i}}x_kg_n(f_1^jz_k)\right\}_{1\leq n\leq K_0,\,0\leq j\leq N_i-1}.$$
It is clear that $\phi$ is a linear mapping from $l_1^{L_{N_i}}$ to $l_{\infty}^{K_{0}\cdot N_i}$ with $||\phi||\le1$.

Next we will show that $\phi(B_1(l_1^{L_{N_i}}))$ contains more than $e^{aN_i}$ points that are $\varepsilon/(9\cdot 2^{K_0})$-separated. Let $E_i$ be an $(N_i,\varepsilon)$-separated subset of $\mathcal{M}(X)$ with
$$\#E_i=s_{N_i}(\mathcal{M}(X),\{f_n\}_{n=1}^{+\infty},\varepsilon)>e^{aN_i}.$$ For any distinct $\mu,\nu\in E_i$, we have $D_{N_i}(\mu,\nu)>\varepsilon$. Thus, there exists $0\leq j_0\leq N_i-1$ such that
\begin{align}\label{eq1}
\sum\limits_{n=1}^{K_0}\frac{\left|\int g_n(f_1^{j_0}x)d\mu(x)-\int g_n(f_1^{j_0}x)d\nu(x)\right|}{2^n}>\frac{\varepsilon}{2}.
\end{align}

We claim that for any distinct $\mu,\nu\in E_i$, the following vectors in $\phi(B_1(l_1^{L_{N_i}}))$ are $\varepsilon/(9\cdot 2^{K_0})$-separated:
$$\phi\left(\mu(A_1),\mu(A_2),\dotsc,\mu(A_{L_{N_i}})\right)\;\;\text{and}\;\;\phi\left(\nu(A_1),\nu(A_2),\dotsc,\nu(A_{L_{N_i}})\right).$$ If the claim is not true, then for any $1\leq n\leq K_0$ and $0\leq j\leq N_i-1$ we have
\begin{align}\label{eq2}
\frac{\left|\sum\limits_{k=1}^{L_{N_i}}\mu(A_k)g_n(f_1^jz_k)-\sum\limits_{k=1}^{L_{N_i}}\nu(A_k)g_n(f_1^jz_k)\right|}{2^n}\leq \frac{\varepsilon}{9\cdot2^{K_0}}.
\end{align}
On the other hand,
\begin{align}\label{eq3}
\left|\int g_n(f_1^jx)d\mu(x)-\int g_n(f_i^jx)d\nu(x)\right|\leq I_1+I_2+I_3,
\end{align}
where
$$I_1=\left|\int g_n(f_1^jx)d\mu(x)-\sum\limits_{k=1}^{L_{N_i}}\mu(A_k)g_n(f_1^jz_k)\right|,$$
$$I_2=\left|\sum\limits_{k=1}^{L_{N_i}}\mu(A_k)g_n(f_1^jz_k)-\sum\limits_{k=1}^{L_{N_i}}\nu(A_k)g_n(f_1^jz_k)\right|$$
and
$$I_3=\left|\int g_n(f_1^jx)d\nu(x)-\sum\limits_{k=1}^{L_{N_i}}\nu(A_k)g_n(f_1^jz_k)\right|.$$
For $k=1,2,\dotsc,L_{N_i}$, if $x\in A_k$ then $\rho(f_1^j(x),f_1^j(z_k))<\delta$ for all $j=0,1,\dotsc,N_i-1$. Thus we have
\begin{align*}
I_1 &=\left|\sum\limits_{k=1}^{L_{N_i}}\int_{A_k}g_n(f_1^jx)-g_n(f_1^jz_k)d\mu(x)\right|\\
 &\leq \sum\limits_{k=1}^{L_{N_i}}\int_{A_k}\left|g_n(f_1^jx)-g_n(f_1^jz_k)\right|d\mu(x)\\
 &\leq\sum\limits_{k=1}^{L_{N_i}}\mu(A_k)\cdot\frac{\varepsilon}{9}\\
 &=\frac{\varepsilon}{9}.
\end{align*}
Similarly, $I_3\leq\varepsilon/9$. By \eqref{eq2}, we know $I_2\leq\varepsilon/9$. So it follows from \eqref{eq3} that
$$\left|\int g_n(f_1^jx)d\mu(x)-\int g_n(f_i^jx)d\nu(x)\right|\leq\varepsilon/3$$ for all $n=1,2,\dotsc,K_0$ and $j=0,1,\dotsc,N_i-1$. This contradicts \eqref{eq1}. Therefore $\phi(\psi(E))\subset\phi(B_1(l_1^{L_{N_i}}))$ are $\varepsilon/(9\cdot 2^{K_0})$-separated, where $$\psi:E\to l_1^{L_{N_i}},\;\;\mu\mapsto\left(\mu(A_1),\mu(A_2),\dotsc,\mu(A_{L_{N_i}})\right).$$

To end the proof, we employ Lemma \ref{lemma1}. In the above discussion, we have shown that $\phi$ is a linear mapping from $l_1^{L_{N_i}}$ to $l_{\infty}^{K_0N_i}$ with $||\phi||\leq 1$ and that $\phi(B_1(l_1^{L_{N_i}}))$ contains more than $e^{aN_i}$ points which are $\varepsilon/(9\cdot 2^{K_0})$-separated. By Lemma \ref{lemma1}, there exist $n_0$ and a constant $c>0$ such that for all sufficiently large $i\in\N$ we have $L_{N_i}\geq 2^{cN_i}$. Thus,

\begin{align*}
h_{top}(X,\{f_n\}_{n=1}^{+\infty})&\ge\lim\limits_{i\to+\infty}\frac{1}{N_i}\log\mathcal{N}\left(\bigvee\limits_{j=0}^{N_i-1}f_1^j\mathcal{U}\right)\\
&\ge\lim_{i\to+\infty}\frac{1}{N_i}\log2^{cN_i}\\&=c\cdot\log2\\&>0.
\end{align*}

\medskip

\section{A constructive proof of Theorem \ref{exam:extension}}\label{sec:exam}
Let $\Sigma_2=\{0,1\}^\N $ and $\sigma: \Sigma_2\to\Sigma_2,\; (a_{n})_{n\in\N}\mapsto(a_{n+1})_{n\in\N}$. For $p\in\N$, $q\in\N_+$ and $i_1,\dotsc,i_q\in\{0,1\}$ we set
$$[i_1,i_2,\dotsc,i_q]_p^q=\left\{(a_{n})_{n\in\N}\in\Sigma_{2}:a_{p+j}=i_{j+1},\forall j=0,1,\dotsc,q-1\right\}.$$

We define
$$X=\{0\}\cup\{1\}\cup\left\{a\times\frac{1}{n}: a\in \Sigma_2, n\in\N_+ \right\},$$ where
$a\times({1}/{n})$ converges to $0$ as $n\to\infty$ for $a\in [0]_0^1$, and $a\times({1}/{n})$ converges to $1$ as $n\to\infty$ for  $a\in [1]_0^1$. We define
$$Y=\{0\}\cup\left\{a\times\frac{1}{n}: a\in\Sigma_2, n\in\N_+ \right\},$$
where $a\times({1}/{n})$ converges to $0$ as $n\to\infty$ for any $a\in \Sigma_2$.

For $n\in\N_+$ we take
$$f_n(x)=\begin{cases}\sigma(a)\times\frac{1}{i+1},&\text{if}\;x=a\times\frac1i\text{ and }i<n,\\x,&\text{otherwise}\end{cases}$$
and let $g_n$ be the restriction of $f_n$ to $Y$. Clearly, $\{f_n\}_{n=1}^{+\infty}$ and $\{g_n\}_{n=1}^{+\infty}$ are sequences of continuous maps on $X$ and $Y$, respectively.

We define a map $\pi:X\to Y$ by
$$\pi(x)=\begin{cases}0,&\text{ if }x=1,\\x,&\text{ otherwise}.\end{cases}$$
We may directly check that $\pi:X\to Y$ is a finite-to-one extension.
Now Theorem \ref{exam:extension} follows from Proposition \ref{pr}.

\begin{prop}\label{pr}
Under the above settings,
$$h_{top}(Y,\{g_n\}_{n=1}^{+\infty})+\sup_{y\in Y}h_{top}(\pi^{-1}(y),\{f_n\}_{n=1}^{+\infty})<h_{top}(X,\{f_n\}_{n=1}^{+\infty}).$$
\end{prop}

\begin{proof}
We first notice that for every $y\in Y$, $h_{top}(\pi^{-1}(y),\{f_n\}_{n=1}^{+\infty})=0$, which means that the second term in the above inequality vanishes. So it remains to deal with the first and third terms. We will show $h_{top}(X,\{f_n\}_{n=1}^{+\infty})>0$ and $h_{top}(Y,\{g_n\}_{n=1}^{+\infty})=0$.

To show $h_{top}(X,\{f_n\}_{n=1}^{+\infty})>0$, we take a finite open cover $\mathcal{U}=\{U_1,U_2\}$ of $X$,
where
$$U_1=\{0\}\cup\left\{a\times\frac{1}{n}: a\in [0]_0^1,n\in\N_+\right\}$$
and
$$U_2=\{1\}\cup\left\{a\times\frac{1}{n}: a\in [1]_0^1,n\in\N_+\right\}.$$
By the construction of $\{f_n\}_{n=1}^{+\infty}$, it is not hard to check that for every $m\in\N_+$ we have
$$f_1^{-m}(\mathcal{U})=\{U_1^m,U_2^m\},$$
where
$$U_1^m=\{0\}\cup\left\{b\times \frac{1}{i}: b_{m-i}=0,i=1,2,\dotsc,m\right\}\cup\left\{b\times \frac{1}{i}: b_1=0,i\geq m+1\right\}$$
and
$$U_2^m=\{1\}\cup\left\{b\times \frac{1}{i}: b_{m-i}=1,i=1,2,\dotsc,m\right\}\cup\left\{b\times \frac{1}{i}: b_1=1,i\geq m+1\right\}.$$
Therefore
$$\mathcal{N}\left(\bigvee_{j=0}^{m-1}f_1^{-j}(\mathcal{U})\right)=2^{m},$$
and thus
\begin{align*}
h_{top}(X,\{f_n\}_{n=1}^{+\infty})&\geq h_{top}(\{f_n\}_{n=1}^{+\infty},\mathcal{U})\\
&=\lim_{N\to+\infty}\frac{\log\bigg(\mathcal{N}\big(\bigvee_{j=0}^{m-1}f_1^{-j}(\mathcal{U})\big)\bigg)}{m}\\
&\geq\lim_{N\to+\infty}\frac{\log2^{m}}{m}\\
&=\log2.
\end{align*}

\medskip
Next we show $h_{top}(Y,\{g_n\}_{n=1}^{+\infty})=0$. Let $\mathcal{V}$ be a finite open cover of $Y$. We choose sufficiently large $N_1,N_2\in\N_+$ such that
$$\mathcal{V^*}=\left\{V_1,V^n_{i_1,i_2,\dotsc,i_{N_2}}: i_1,i_2,\dotsc,i_{N_2}\in\{0,1\},1\leq n\leq N_1\right\}$$
is a refinement of $\mathcal{V}$, where
$$V_1=\{0\}\cup\left\{a\times\frac{1}{n}: a\in\Sigma_2,n>N_1\right\}$$ and
$$V^n_{i_1,i_2,\dotsc,i_{N_2}}=\left\{a\times\frac{1}{n}:a\in[i_1,i_2,\dotsc,i_{N_2}]_1^{N_2}\right\}$$ for all $i_1,i_2,\dotsc,i_{N_2}\in\{0,1\}$ and $1\leq n\leq N_1$. By the definition of $g_{n}$, for every $x\in Y$ and every integer $n>N_1$ we have $g_1^nx\in V_1$, that is, $g_1^{-n}(\mathcal{V^*})=\{Y,\emptyset\}$. Thus,
$$\mathcal{N}\left(\bigvee_{i=0}^{n-1}g_1^{-i}(\mathcal{V}^*)\right)=\mathcal{N}\left(\bigvee_{i=0}^{N_1}g_1^{-i}(\mathcal{V}^*)\right)$$
for all $n>N_1$.
Therefore,
\begin{align*}
h_{top}(\mathcal{V},\{g_n\}_{n=1}^{+\infty})&\le h_{top}(\mathcal{V^*},\{g_n\}_{n=1}^{+\infty})
=\lim_{n\to+\infty}\frac{\log\mathcal{N}\left(\bigvee_{i=0}^{N_1}g_1^{-i}(\mathcal{V}^*)\right)}{n}=0.
\end{align*}
Since $\mathcal{V}$ is arbitrary, we see that $h_{top}(Y,\{g_n\}_{n=1}^{+\infty})=0$.
\end{proof}

\bibliographystyle{amsplain}

\end{document}